\documentclass[reqno]{amsart}     

\usepackage{amsmath,amssymb}    
\usepackage{mathrsfs}    
\usepackage{verbatim}
\usepackage{color} 
\usepackage[backref,colorlinks]{hyperref}  

\newtheorem{theorem}                   {Theorem}
 
\newtheorem{lemma}           [theorem] {Lemma}  
   
\newtheorem{corollary}       [theorem] {Corollary}   
   
\newtheorem{proposition}     [theorem] {Proposition}

\theoremstyle{remark}

\newcommand{\eps}{\varepsilon}

\newcommand{\dcup}{\dot{\cup}}

\begin{document}

\title[Multicolor Ramsey Numbers]{On the multicolor Ramsey number of a graph with $m$ edges}
\author[Kathleen Johst]{Kathleen Johst*}
\address{Freie Universit\"at Berlin, Institut f\"ur Mathematik,  
 Berlin, Germany}
\email{kathleen.johst@fu-berlin.de}
\author[Yury Person]{Yury Person}
\address{Goethe-Universit\"at, Institut f\"ur Mathematik,
  Robert-Mayer-Str. 10, 60325 Frankfurt am Main, Germany}
\email{person@math.uni-frankfurt.de}

\thanks{
* This paper forms part of the first author's Master's thesis  at Freie Universit\"at Berlin.
}

\date{\today}

\begin{abstract}
  The multicolor Ramsey number $r_k(F)$ of a graph $F$ is the least integer $n$ such that 
in every coloring of the edges of $K_n$ by $k$ colors there is a monochromatic copy of $F$. In this short note 
we prove an upper bound on $r_k(F)$ for a graph $F$ with $m$ edges and no isolated vertices of the form $k^{6km^{2/3}}$ addressing a question of Sudakov [ Adv. Math. 227 (2011), no. 1,
601--609]. 
Furthermore, the constant in the exponent in the case of bipartite $F$ and two colors is lowered so that  $r_2(F)\le 2^{(1+o(1))2\sqrt{2m}}$ 
improving the result of Alon, Krivelevich and Sudakov [Combin. Probab. Comput. 12 (2003), no. 5--6, 477--494].
\end{abstract}

\maketitle

\section{Introduction}
 The by now classical theorem of Ramsey~\cite{Ram30} states that no matter how one colors 
the edges of the large enough complete graph $K_n$ with two colors, say red and blue, there will always be a 
monochromatic copy of $K_t$ in it. The smallest such $n$ is called the Ramsey number, denoted by $r(t)$  or $r(K_t)$. 
 First lower and upper bounds on $r(t)$ were obtained by Erd\H{o}s and Szekeres~
\cite{ErdSze35}: $r(t)\le \tbinom{2t-2}{t-1}$ and by 
 Erd\H{o}s~\cite{Erd47}: $r(t)\ge 2^{t/2}$. 
 Despite numerous efforts by various researchers, the best lower and upper bounds remain asymptotically $2^{(1+o(1))t/2}$ and $2^{(1+o(1))2t}$, for the 
currently best bounds see Conlon~\cite{Con09} and Spencer~\cite{Spe75}. 

Thus,  one turned to the study of Ramsey numbers of graphs other than complete graphs $K_t$. 
The multicolor Ramsey number for $k$ colors of a graph $F$, denoted $r_k(F)$, is defined as the smallest number $n$ such that in any coloring of $E(K_n)$  by $k$ colors there is 
a monochromatic copy of a graph $F$ in one of the $k$ colors. Much attention was drawn by the conjectures of Burr and Erd\H{o}s~\cite{BurErd75} about Ramsey numbers of graphs $F$ whose maximum degree is 
bounded by a constant and which are $d$-degenerate for some constant $d$ stating that these Ramsey numbers are linear in $v(F):=|V(F)|$.
 While the first conjecture has been resolved positively by Chvat\'al, R\"odl, Szemer\'edi and Trotter~\cite{CRST83}, the latter one is still open 
and the best bound is due to Fox and Sudakov~\cite{FoxSud09BE} being $r_2(F)\le 2^{c_d\sqrt{\log n}}n$ for $c_d$ depending on $d$ only.  

A related conjecture of Erd\H{o}s and Graham~\cite{ErdGra75}  states that among all graphs $F$ with $m=\tbinom{t}{2}$ edges 
and no isolated vertices  
 the Ramsey number $r(t)$ of the complete graph $K_t$ is an upper bound on $r_2(F)$.  
A relaxation conjectured by Erd\H{o}s~\cite{ChuGra98} states that at least $r(F)\le 2^{c\sqrt{m}}$ should hold for any graph $F$ with $m$ edges and no isolated vertices and some absolute constant $c$. 
This was verified by  Alon, Krivelevich and Sudakov~\cite{AKS03} who showed that if $F$ is bipartite, has $m$ edges and  no isolated vertices then 
 $r(F)\le 2^{16\sqrt{m}+1}$, 
and for nonbipartite $F$  showing $r(F)\le 2^{7\sqrt{m}\log_2 m}$. Finally, the general case was settled by Sudakov~\cite{Sud11} who proved $r(F)\le 2^{250\sqrt{m}}$.

In his concluding remarks in~\cite{Sud11}, Sudakov mentions that the methods used to settle the 
general case are not extendible to more colors and it would be interesting to understand 
the growth of $r_k(F)$. It is easy to see that there is an upper bound on $r_k(F)$ of the form $k^{k v(F)}$  by 
finding a monochromatic copy of $K_{2m}\supset F$ using the classical color focussing argument.
 In this note we prove to the best of our knowledge a first nontrivial upper bound  $r_k(F)\le k^{6km^{2/3}}$.
\begin{theorem}\label{thm:multi}
 Let $F$ be a graph with $m$ edges and no isolated vertices. Then, for $k\ge 3$  it holds  
\[
 r_k(F)\le k^{3 \cdot2^{-1/3}km^{2/3}+k(2m)^{1/3}}8m. 
\]
\end{theorem}

Further we study the case when $F$ is bipartite and show an upper bound 
$r_k(F)\le  k^{(1+o(1))2\sqrt{mk}}$. 
\begin{theorem}\label{thm:bip}
 Let $F$ be a bipartite graph with $m$ edges and no isolated vertices. Then, for $k\ge 2$  it holds  
\[
  r_k(F)\le 2^6 m^{3/2} k^{2\sqrt{km}+1/2}.
\]
\end{theorem}
Note that in the case $k=2$, Theorem~\ref{thm:bip} is an improvement of the above mentioned result 
of Alon, Krivelevich and Sudakov~\cite{AKS03} to $r(F)\le 2^{(1+o(1))2\sqrt{2m}}$. 
Remarkably, this upper bound is asymptotically the ``same'' as the upper bound $2^{(1+o(1))2k}$ for $r(k)$ with $m=\tbinom{k}{2}$. 

The methods we use are  slight modifications of the arguments of Fox and Sudakov from~\cite{FoxSud09a} and of Alon, Krivelevich and Sudakov~\cite{AKS03}. 
The paper is organized as follows. In the next section, Section~\ref{sec:tools} we collect some results and observations used in our proofs, in 
Section~\ref{sec:bip} we prove Theorem~\ref{thm:bip} and in Section~\ref{sec:multi} we show Theorem~\ref{thm:multi}.

\section{Some auxiliary results}\label{sec:tools}
Here we collect several results from~\cite{FoxSud09a} and one small graph theoretic estimate. 
The first prominent lemma we use is the so-called dependent random choice lemma, stating 
that in a bipartite dense graph one finds a large vertex subset in one class, 
most of whose $d$-tuples have many common neighbours on the other side.

\begin{lemma}[Dependent Random Choice, Lemma~2.1\cite{FoxSud09a}]\label{lem:drc}
If $\eps > 0$ and $G = (V_1, V_2; E)$  is a bipartite graph with $|V_1| = |V_2| = N$ and at least $\eps N^2$ edges, then for all 
positive integers $a$, $d$, $t$, $x$, there is a subset $A \subset V_2$ with $|A| \geq 2^{- \frac{1}{a}} \eps^t N$, such 
that for all but at most $2 \eps^{-ta} \left(\tfrac{x}{N} \right)^t \left(\tfrac{|A|}{N} \right)^a \binom{N}{d}$ $d$-sets $S$ in $A$, we have $|N(S)| \geq x$.
\end{lemma}

The following lemma allows one to embed a graph $H$ with bounded degree and bounded chromatic number into a dense graph $G$ given along with a nested sequence of sets, where 
the parts of $H$ are supposed to be embedded into. 

\begin{lemma}[Lemma~4.2 in~\cite{FoxSud09a}]\label{lem:embed}
Suppose $G$ is a graph with vertex set $V_1$, and let $V_1 \supset \ldots \supset V_q$ be a family of nested subsets 
of $V_1$ such that $|V_q| \geq x \geq 4n$, and for $1 \leq i <q$, all but less than $(2d)^{-d} \binom{x}{d}$ $d$-sets 
$U \subset V_{i+1}$ satisfy $|N(U) \cap V_i| \geq x$. Then, for every $q$-partite graph $H$ with $n$ vertices and maximum 
degree at most $\Delta(H) \leq d$, there are at least $\left( \tfrac{x}{4} \right)^n$ labeled copies of $H$ in $G$. 
\end{lemma}

We will also need the following Tur\'an-type result, from which the currently best known upper bound on the Ramsey number of a bounded degree bipartite graph follows.
\begin{theorem}[Theorem~1.1 from~\cite{FoxSud09a}]\label{thm:densityresult}
Let $H$ be a bipartite graph with $n$ vertices and maximum degree $\Delta \geq 1$. 
If $\eps > 0$ and $G$ is a graph with $N \geq 32 \Delta \eps^{- \Delta} n$ vertices and at 
least $\eps \binom{N}{2}$ edges, then $H$ is a subgraph of $G$.
\end{theorem}

Finally we need one simple observation, whose proof we provide here for completeness.
\begin{proposition}\label{prop:bddeg}
 Let $F=(V,E)$ be a graph with $m$ edges. Then there exists a subset $U\subseteq V$ with $|U|< \tfrac{m}{d}$ such that $\Delta(F\setminus U)\le d$.
\end{proposition}
\begin{proof}
Let $v_1$ be a vertex of maximum degree in $F_1:=F$ and set $d_1:=\Delta(F)$. 
We delete $v_1$ from $F$ denoting the new graph by $F_2$. We proceed inductively,
deleting from $F_i$ a vertex of maximum degree $v_{i}$, setting $d_{i}:=\Delta(F_i)$ 
and defining the new graph $F_{i+1}:=F_i-v_{i}$ and stop with $F_{|V(F)|+1}=\emptyset$. 
Let $j$ be the smallest integer with $\Delta(F_{j+1})\le d$. 
Thus, till we obtained $F_{j+1}$ we must have deleted $j$ vertices, each of degree larger than $d$. 
Moreover, by the construction of the sequence of $v_i$s, 
we have $m=|E(F)|=\sum_{i=1}^{|V(F)|}\Delta(F_{i})$. Therefore, $jd< m$ and the claim follows with $U:=\{v_1,\ldots,v_j\}$.
\end{proof}

Often we try to avoid using  floor and ceiling signs as they will not affect our calculations.

\section{Multicolor Ramsey number of bipartite graphs with \texorpdfstring{$m$}{m} edges}\label{sec:bip}
The idea of the proof of Theorem~\ref{thm:bip} is quite simple. Given a coloring of $E(K_N)$, we will perform a color focussing argument by 
considering the densest color class and taking a vertex with maximum degree in it. 
Then we  iterate on the colored neighborhood of that vertex. 
 After less than $km/d$ steps we arrive at the situation, where we can embed 
 all $m/d$  vertices from $U$ (of high degree in $F$) onto the vertices specified in 
the focussing process, the remaining graph $F-U$ has maximum degree at most $d$ (by Proposition~\ref{prop:bddeg}) and is bipartite, 
and thus can be embedded in one round, by Theorem~\ref{thm:densityresult}.
\begin{proof}[Proof of Theorem~\ref{thm:bip}]
Given a bipartite graph $F$ with $m$ edges and no isolated vertices. We choose with foresight $d=\sqrt{km}$. 
By Proposition~\ref{prop:bddeg}, let $U$ be a set of $t=\lfloor m/d\rfloor=\lfloor \sqrt{m/k}\rfloor$ vertices 
such that $\Delta(F\setminus U)\le d$. Further observe that $|V(H)|\le 2m$. 

Let us be given an arbitrary but fixed $k$-edge coloring of the graph $G:=K_N$ with the colors $1$, \ldots, $k$, where $N = 32dk^{d+kt}2m$. 

We will construct a sequence of sets $A_1\supset A_2\supset \ldots\supset A_s$ and a sequence of colors $c(1)$,\ldots, $c(s)$ as follows. 
First we set $A_1=[N]$ and let $c(1)$ be a densest color in $G[A_1]$. 
Since we used $k$ colors there exists a vertex $v_1\in A_1$ connected to at least   
 $\tfrac{|A_1|-1}{k}$ vertices in color $c(1)$. We denote the set of these vertices by $A_2$. Then we proceed inductively as follows. 
Given a sequence of sets $A_1\supset A_2\supset \ldots\supset A_{i+1}$ and the corresponding sequences of vertices $v_1$,\ldots, $v_i$ and colors 
$c(1)$,\ldots, $c(i)$, we do the following. Let $c(i+1)$ be the densest color in $G[A_{i+1}]$. If $c(i+1)$ occurs at most $t$ times in the sequence $c(1)$,\ldots, $c(i+1)$ of colors constructed so far, 
then we choose $v_{i+1}\in A_{i+1}$ such that $v_{i+1}$ is connected to at least $\tfrac{|A_{i+1}|-1}{k}$ vertices in color $c(i+1)$ and we denote these vertices by $A_{i+2}$.
Otherwise we stop. It is clear that we stop after at most $kt+1$ steps, that is with $i+1\le kt+1$.

Next we identify $t$ vertices  $v_{j_1}$,\ldots, $v_{j_t}$ such that $c(j_1)=\ldots=c(j_t)=c(i+1)$. Observe that all vertices $v_{j_s}$ are connected in color $c(i+1)$ to each other 
and also to all vertices in $A_{i+1}$. Moreover, at least $\tfrac{1}{k}\tbinom{|A_{i+1}|}{2}$ edges of 
$G[A_{i+1}]$ are colored in $c(i+1)$. Therefore, we can embed the vertices from $U$ in $F$ onto 
$v_{j_1}$, \ldots, $v_{j_t}$, and then one needs to find an embedding 
$F\setminus U$ into $G[A_{i+1}]$ in color $c(i+1)$. 
But this is asserted to us by Theorem~\ref{thm:densityresult}, as long as 
\[
 |A_{i+1}|\ge 32 d k^d |V(F)\setminus U|.
\]

Since $i+1\le kt+1$ we obtain $|A_{i+1}|\ge \tfrac{N}{k^{kt}}-1$, and since we can assume 
that $F$ is not a matching (otherwise Theorem~\ref{thm:densityresult} implies the result immediately), 
we have $|V(F)|< 2m$ and therefore
\[
   |A_{i+1}|\ge 32 d k^d 2m-1\ge 32 d k^d |V(F)\setminus U|.
\]
Thus, $r_k(F)\le N= 32dk^{d+kt}2m\le 2^6 \sqrt{km^3} k^{2\sqrt{km}}=k^{(1+o(1))2\sqrt{km}}$. 
 \end{proof}

As an immediate consequence we get.
\begin{corollary}
Let $F$ be a bipartite graph with $m$ edges and without isolated vertices, then $r_2(F) \leq 2^{(1+o(1))2 \sqrt{2m}}$. 
\end{corollary}

\section{Multicolor Ramsey number of general graphs with \texorpdfstring{$m$}{m} edges}\label{sec:multi}
In this section we heavily rely on the tools developed by Fox and Sudakov in~\cite{FoxSud09a}. There they showed that $r_k(F)\le k^{2k\Delta q}n$ for  a 
graph $F$ with $n$ vertices, $\Delta(F)=\Delta$ and $\chi(F)=q$ (more generally, it holds for $k$ not necessarily isomorphic graphs $F_1$, \ldots, $F_k$ with the same properties). 
Their proof combines Lemmas~\ref{lem:drc} and~\ref{lem:embed}. 

Our proof strategy is in fact a slight modification of their argument 
intertwined with the process of  first embedding high degree vertices. The idea of embedding high degree vertices already occurs in~\cite{AKS03}. 
More precisely, since we are given a general graph $F$ with $m$ edges, we first seek to embed vertices of high degree (which will be done 
in a similar way as in the proof of Theorem~\ref{thm:bip}). 
However, this time we are going to use Lemma~\ref{lem:drc} instead of Theorem~\ref{thm:densityresult} repeatedly. 
The authors in~\cite{FoxSud09a}  show $r_k(F)\le k^{2k\Delta q}n$ by  applying 
iteratively Lemma~\ref{lem:drc} roughly $qk$ times, ``loosing'' each time roughly a factor of $k^{-k\Delta}$. 
Afterwards one identifies a long enough nested sequence to perform embedding 
(Lemma~\ref{lem:embed}). In our case however, we first need to reduce the maximum degree of $F$, 
and only then we will apply Lemma~\ref{lem:drc}. However, its applications intertwine with the 
focussing argument similar to the previous section, as each color might get filled up differently quickly.

\begin{proof}[Proof of Theorem~\ref{thm:multi}]
We choose with foresight $d=(m/4)^{1/3}$ and $\ell=\lfloor m/d\rfloor=\lfloor (2m)^{2/3}\rfloor$.  
Furthermore, we set $x = k^{-(2d+2)kd -k\ell}N$ and $N= k^{(2d+2)kd + k\ell} 8m$. 
Take a given graph  $F$ with $m$ edges and no isolated vertices.  
By Proposition~\ref{prop:bddeg}, let $U$ be a set of at most $\ell$ vertices such that $\Delta(F\setminus U)\le d$. Further observe that 
$|V(H)|\le 2m$ and $\chi(F\setminus U)\le d+1$. 

Let  an arbitrary but fixed coloring of the edges of the graph  $G:=K_N$ by $k$ colors be given. 

We set $A_1=[N]$ and construct a sequence of sets $A_1\supset A_2\supset \ldots\supset A_s$ and a sequence of colors $c(1)$,\ldots, $c(s-1)$ inductively as follows.
 

Given a sequence of sets $A_1\supset A_2\supset \ldots\supset A_{i}$ and the sequence of colors 
$c(1)$,\ldots, $c(i-1)$, we do the following. Let $c(i)$ be the densest color in $G[A_{i}]$. 
If $c(i)$ occurs at most $\ell$ times in the sequence 
$c(1)$,\ldots, $c(i)$ of colors constructed so far, 
then we choose $v_{i}\in A_{i}$ such that $v_{i}$ is connected to at least $\tfrac{|A_{i}|-1}{k}$ vertices 
in color $c(i)$ and we denote these vertices by $A_{i+1}$ (and we refer to this step as \emph{focussing}). 
If, however, the color $c(i)$ occurs more than $\ell$ times then we call $c(i)$ \emph{saturated}. 
As long as the saturated color $c(i)$ occurs at most  $t+d$ times among $c(1)$, \ldots, $c(i)$, 
 we consider a balanced bipartition of $A_{i}=A_{i,1}\dcup A_{i,2}$ (assume $|A_{i,1}|\le |A_{i,2}|$) such that 
at least $\tfrac{1}{k}|A_{i,1}||A_{i,2}|$ edges are colored by the color $c(i)$ (simply take a random balanced bipartition). 
Furthermore, we apply now Lemma~\ref{lem:drc} with $\eps=\tfrac{1}{k}$, 
$a=1$ and $t=2d$ and thus we find a subset $A_{i+1}\subset A_{i,2}\subset A_i$ with $|A_{i+1}|\ge 2^{-1} k^{-2d} |A_{i,2}|\ge k^{-2d-2}|A_i|$ (use $|A_{i,1}|\le |A_{i,2}|$),
 such that  all but at most 
\begin{equation}\label{eq:badsets}
 2 \cdot k^{2d} \left( \frac{x}{|A_{i,2}|} \right)^{2d} \left(\frac{|A_{i+1}|}{|A_{i,2}|} \right) \binom{|A_{i,2}|}{d}
\end{equation}
$d$-sets $S$ in $A_{i+1}$ have at least $x$ common neighbors in $G[A_i]$ in color $c(i)$ (actually in $A_{i,1}$). We refer to  such a step as \emph{nesting}. 
Moreover, we can use $|A_{i,2}|\ge \tfrac{1}{2}k^{-k\ell-(2d+2)k(d-1)} N$ to simplify and  bound~\eqref{eq:badsets} as: 
\[
 \left(\frac{x}{|A_{i,2}|}\right)^d (2k^2)^d  \binom{x}{d} \le \left(2 k^{-(2d+2)k}\right)^d (2k^2)^d  \binom{x}{d} \le (2d)^{-d} \binom{x}{d}.
\]
We stop constructing a sequence once we end up with colors $c(1)$, \ldots, $c(s)$ and sets 
$A_1\supset\ldots\supset A_{s+1}$ and there is one color $c$ which occurs 
$\ell+d$ times. Clearly, $s\le kt+k(d-1)+1$, since we first \emph{focus} in one color $t$ times before it gets saturated and then we need to \emph{nest} $d$ times in some color, 
before we stop the sequence construction. By the choice of  $N$, $x$, $\ell$ and $d$ we can clearly proceed for $k\ell+ k (d-1)+1$ steps if necessary. 

Let $c\in[k]$ be the color which occurs $\ell+d$ times and let $v_{i_1}$, \ldots, $v_{i_\ell}$ be the vertices 
which got selected in the first $\ell$ steps when the color $c$ was chosen (these were focussing steps) and 
let $A_{i_{\ell+1}}$,\ldots, $A_{i_{\ell+d}}$ be the sets with majority color $c$ during the nesting steps.  
Next we show how to find a $c$-colored copy of $F$, whose vertices are embedded onto $v_{i_1}$, \ldots, $v_{i_\ell}$ and 
into the sets  $A_{i_{\ell+1}}$, $A_{i_{\ell+1}+1}$,\ldots, $A_{i_{\ell+d}+1}$ (where $i_{\ell+d}=s$). 

By Proposition~\ref{prop:bddeg}, we have a set $U$ of  $\ell$ vertices of high degree 
which get embedded onto $v_{i_1}$, \ldots, $v_{i_\ell}$ and then all we 
need to do is to embed a copy of $F\setminus U$ into the nested sequence $A_{i_{\ell+1}}$, $A_{i_{\ell+1}+1}$,\ldots, $A_{i_{\ell+d}+1}$ 
in color $c$. But this can be done by Lemma~\ref{lem:embed}. This shows that $r_k(F)\le N$.
\end{proof}

\begin{corollary}
Let $F$ be a graph with $m$ edges and without isolated vertices. Then $r_k(F) \leq k^{6km^{2/3}}$.
\end{corollary}

\section{Concluding Remarks}
In this note we showed a first nontrivial upper bound on $r_k(F)$ to be $k^{6ke(F)^{2/3}}$. 
Certainly, there should be a room for improvement, maybe even to $k^{O(k \sqrt{e(F)})}$ thus generalizing the result of Sudakov. 
Another interesting direction would be to improve the result of Sudakov to $r_2(F)\le 2^{(1+o(1))2\sqrt{2m}}$ by obtaining  asymptotically 
the same upper bound as the best one known for $r_2(K_t)$ with $\tbinom{t}{2}=m$. This was noted by us to hold if $F$ is bipartite.
 
After the completion of this paper we learned that Conlon, Fox and Sudakov obtained 
a  result similar to our Theorem~\ref{thm:multi} independently~\cite{Sud13}.

\end{document}